\documentclass[%
aip,
amsmath,amssymb,
reprint, onecolumn 
]{revtex4-1}
\usepackage{amsthm}
\usepackage{amssymb}
\usepackage{amsmath}
\usepackage{graphicx}
\usepackage{dcolumn}
\usepackage{bm}
\usepackage{enumerate}
\usepackage[utf8]{inputenc}
\usepackage[T1]{fontenc}
\usepackage{mathptmx}
\usepackage{etoolbox}
\newtheorem{thm}{Theorem}[section]

\newtheorem{lem}[thm]{Lemma}
\newtheorem{prop}[thm]{Proposition}

\newtheorem{defn}[thm]{Definition}

\newcommand{\mb}{\mathbb}

\newcommand{\lk}{\left}
\newcommand{\re}{\right}

\def \a{\alpha} \def \b{\beta}  \def \d{\delta}
\def \t{\theta} \def \T{\Theta}  \def \e{\epsilon} 
   \def \o{\omega}
\def \O{\Omega}

\makeatletter
\def\@email#1#2{%
	\endgroup
	\patchcmd{\titleblock@produce}
	{\frontmatter@RRAPformat}
	{\frontmatter@RRAPformat{\produce@RRAP{*#1\href{mailto:#2}{#2}}}\frontmatter@RRAPformat}
	{}{}
}%
\makeatother
\begin{document}
	
	\preprint{AIP/123-QED}
	
	\title[Directional Pinsker algebra and its applications]{Directional Pinsker algebra and its applications}
	\author{Chunlin Liu}

	\author{Leiye Xu}
	\affiliation{CAS Wu Wen-Tsun Key Laboratory of Mathematics, School of Mathematical Sciences, University of Science and Technology of China, Hefei, Anhui, 230026, PR China
	}%
	
	\email{lcl666@mail.ustc.edu.cn,leoasa@mail.ustc.edu.cn}	\date{\today}
	
	\begin{abstract}
		In this paper, we introduce the directional Pinsker algebra, and construct a skew product to study it.  As applications, 
		we show that 
		\begin{enumerate}[(i)]
			\item 	if a $\mathbb{Z}^2$-system  with positive directional measure-theoretic entropy  then it is multivariant directional mean Li-Yorke chaotic along the corresponding direction;
			
			\item  for any ergodic   measure on a $\mathbb{Z}^2$-system, the intersection of the set of directional measure-theoretic entropy tuples with the set of directional asymptotic tuples is dense in the set of directional measure-theoretic entropy tuples. 
		\end{enumerate}
	\end{abstract}
	
	\maketitle
	\section{Introduction}
	Cellular automaton constitute a collection of discrete computational models that define their evolution through neighborhood rules. They have many applications in simulating various complex systems in physics and science, and also have been explored in the context of statistical mechanics, primarily for investigating phase transitions \cite{li1990transition} and the Ising model \cite{vichniac1984simulating}. Overall, cellular automaton have been used to model physical systems. In particular, the thermodynamics of cellular automaton has already captured the attention of numerous scholars. For instance, to investigate the cellular automaton map together with the Bernoulli shift, Milnor \cite{M1,M2} defined directional entropy. After that, Park \cite{Park1}  showed that the directional entropy is continuous for a $\mathbb{Z}^2$-action generated by a cellular automaton. In fact, Park proved this result holds on more general systems.
	With the deepening of research, mathematicians
	are not only focused on the directional entropy of cellular automaton but on more general systems as well (see \cite{MR3551897,MR3681987,LXETDS,MR1355676,MR2755932,MR3759528}).
	
	It is well known that for a measure theoretical dynamical system under group actions,  there exists the largest sub-$\sigma$-algebra, called Pinsker algebra, with zero entropy. It is an indispensable tool in the study of entropy. However, there is no a similar conception for directional dynamical systems. 
	In this paper, we introduce directional Pinsker algebra.
	Moreover, we construct a skew product system in terms of this measure-theoretical system and some direction, and show that the directional entropy of original system is equal to the entropy  and fiber entropy of this skew product system (\textbf{see Theorem \ref{lem1}}). Moreover, we establish the relation between directional Pinsker algebra and the Pinsker algebra of this skew product system  (\textbf{see Theorem \ref{thm2}}). These are our main theorem in this paper. Based the second consequence, many results of classical Pinsker algebra under the group actions can be extended to directional Pinsker algebra, which  overcomes the difficulty that a  dynamical system from a direction viewpoint is not  a dynamical system under group action. 
	
	\smallskip
	As a form of chaos, positive entropy has some  connections with other chaotic behaviors. In particular, Blanchard, Glasner, Kolyada, and Maass \cite{MR1900136} proved that  positive entropy implies Li-Yorke chaos.  Downarowicz \cite{MR3119189} observed that mean Li-Yorke chaos is equivalent to  DC2 chaos and proved that positive entropy implies mean Li-Yorke chaos. Recently, Huang, Li and Ye \cite{MR4397149} showed that  positive entropy implies Li-Yorke chaos along any infinite sequence. The more results on Li-Yorke chaos can be seen \cite{MR3570021,MR2317754,LQ,liu2022pinsker}.  The reader also can refer to the survey \cite{MR3431162} for more details about chaos. From the results above-mentioned, chaotic behavior usually occurs in the systems with positive entropy. It is easy to construct a  $\mathbb{Z}^2$-system with zero entropy but positive directional entropy. Thus,  there is a natural question:
	is there  Li-Yorke chaos phenomenon present in systems of this kind as well? As the first application of our main result, we  provide a positive answer.
	\\
	\textbf{Application 1 (Theorem \ref{m-thm1}):}
	If a system  with positive directional measure-theoretic entropy  then it is multivariant directional mean Li-Yorke chaotic along the corresponding direction.
	\smallskip
	
	To further study the chaos in dynamical systems with positive measure-theoretical entropy.  We introduce and investigate the directional stable and unstable sets. Moreover, using localization theory ideas, we introduce the directional entropy tuples, and study some properties of them. With the help of these consequences and our main theorems, we prove our second main application:\\
	\textbf{Application 2 (Theorem \ref{m-thm2}):}
	For any ergodic  $\mathbb{Z}^2$-system, the intersection of the set of directional measure-theoretic entropy tuples with the set of directional asymptotic tuples is dense in the set of directional measure-theoretic entropy tuples.

	\medskip
	\medskip
	This paper is organized as follows. In Section 2, we recall some basic notions that we use in this paper. In Section 3, we introduce directional Pinsker algebra, and construct a skew product system to study directional entropy and Pinsker algebra. In Section 4, we prove the first application. In Section 5,
	we introduce directional entropy n-tuples and study many properties of them. In Section 6, we prove  the second application.
	
	\section{Preliminaries}
	In this section we recall some notions of dynamical systems that are used later. 
	Throughout this paper,  by a $\mathbb{Z}^2$-topological dynamical system ($\mathbb{Z}^2$-t.d.s. for short), we mean a pair $(X,T)$, where $X$ is a compact metric space with a metric $d$ and the $\mathbb{Z}^2$-action $T:X\to X$ is a homeomorphism from the additive group $\mathbb{Z}^2$ to the group of homeomorphisms of $X$ such that $T^{\vec{v}}\circ T^{\vec{w}}=T^{\vec{v}+\vec{w}}$ for any $\vec{v},\vec{w}\in \mathbb{Z}^2 $ and $T^{\vec{0}} $ is the identity on $X$.  Let $\mathcal{B}_X$ be the Borel $\sigma$-algebra of $X$ and  $M(X)$ be the set of Borel probability measures defined on $\mathcal{B}_X$. The support of $\mu \in M(X)$, denoted by $\operatorname{supp}\mu$ is defined to be the set of all points $x\in X$ for which every open neighborhood $U$ of $x$ has positive measure.
	We say that $\mu\in M(X)$ is $T$-invariant if for any $\vec{v}\in \mathbb{Z}^2$ and $B\in\mathcal{B}_X$, $\mu (T^{-\vec{v}}B)=\mu(B)$. Denote by $M(X,T)$ the set of all $T$-invariant measure. We say $\mu\in M(X,T)$ is ergodic if each $B\in\mathcal{B}_X$ with $T^{-\vec{v}}B=B$ for all $\vec{v}\in\mathbb{Z}^2$ implies that either $\mu(B)=0$ or $1$. Denote by $M^e(X,T)$ the set of all ergodic measures.   Meanwhile, each $\mu\in M(X,T)$ induces a $\mathbb{Z}^2$-meausre preserving dynamical system ($\mathbb{Z}^2$-m.p.s. for short) $(X,\mathcal{B}_X,\mu,T)$.
	\smallskip
	
	Let  $\vec{v}=(1,\beta)\in \mb{R}^2$ be a direction vector.  We remark that, in this paper, we only consider the case $\b\notin\mathbb{Q}$,  as  the case that $\vec{v}=(0,1)$ or $\b\in\mathbb{Q}$, can be transformed to a $\mathbb{Z}$-action.  Given $b\in(0,\infty)$,
	we put $$\Lambda^{\vec{v}}(b)=\left\{(m,n)\in\mathbb{Z}^2:\beta m-b\leq n\le \beta m+b\right\},$$  and we write $\Lambda^{\vec{v}}(b)\cap ([0,N-1]\times \mathbb{Z})$ as $\Lambda_N^{\vec{v}}(b)$ for any $N\in\mathbb{N}$.

	\subsection{Entropy for $\mathbb{Z}$-actions}Let $(X,T)$ be a $\mathbb{Z}$-t.d.s.
	We denote the collection of finite partitions and finite open covers of $X$ by $\mathcal{P}_{X}$ and $\mathcal{C}_{X}^{o}$, respectively. Given $\mathcal{U},\mathcal{V}\in \mathcal{C}_{X}^{o}$, define $\mathcal{U} \vee \mathcal{V}:=\{U \cap V: U \in \mathcal{U}, V \in \mathcal{V}\}$. Similarly, we can define $\a\vee\b$ for $\a,\b\in\mathcal{P}_X$.
	
	Define $N(\mathcal{U})$ as the minimum among the cardinalities of the subcovers of $\mathcal{U}\in \mathcal{C}_{X}^o$. The topological entropy of $\mathcal{U}$ with respect to $T$ is defined by
	$$
	h_{top}(T, \mathcal{U}):=\lim _{n \to\infty} \frac{1}{n} \log N\left(\bigvee_{i=0}^{n-1} T^{-i} \mathcal{U}\right) .
	$$
	The topological entropy of $(X, T)$ is defined by $h_{top}(T):=\sup _{\mathcal{U} \in \mathcal{C}_{X}^{\circ}} h_{top}(T, \mathcal{U})$.
	
	For a given $\alpha \in \mathcal{P}_{X}$ and $\mu \in M(X)$, let $H_{\mu}(\alpha)=\sum_{A \in \alpha}-\mu(A) \log \mu(A)$.  
	When $\mu \in M(X, T)$, we define the measure-theoretic entropy of $\alpha$ with respect to $\mu$ and $T$ as
	$$
	h_{\mu}(T, \alpha):=\lim _{n \rightarrow \infty} \frac{1}{n} H_{\mu}\left(\bigvee_{i=0}^{n-1} T^{-i} \alpha\right) .
	$$
	The measure-theoretic entropy of $(X,\mathcal{B}_X,\mu,T)$ is defined by $h_{\mu}(T):=\sup _{\alpha \in \mathcal{P}_{X}} h_{\mu}(T, \alpha)$. 
	The relation between topological entropy and measure-theoretic entropy is the following well known variational principle: $h_{top}(T)=\sup _{\mu \in M(X, T)} h_{\mu}(T)=\sup _{\mu \in M^{e}(X, T)} h_{\mu}(T)$ (see for example \cite[Chapter 9.3]{Peter}).
	
	Let $(\O,\mathcal{F},P,\theta)$ be a $\mathbb{Z}$-m.p.s., and $\{\varphi(\o):\o\in\O\}$ be a family of measurable transformations on the space $(X,\mathcal{B}_X)$ such that the map $(\o,x)\mapsto \varphi(\o)x$ is measurable from $\O\times X$ to $X$. Then the skew-product transformation $\T$ on $\O\times X$ can be defined by 
	\[\T(\o,x)=(\t\o,\varphi(\o)x)\text{ for any }(\o,x)\in \O\times X.\]
	Given  $\mu\in M(X)$ such that it is invariant with respect to $\varphi(\o)$ for all $\o\in\O$, then the product measure $\widetilde{\mu}=P\times \mu$ is $\T$-invariant. The fiber entropy of $\varphi$ with respect to the partition $\a\in\mathcal{P}_{X}$ can be defined via
	\[h_{\widetilde{\mu}}(\varphi,\a):=\lim_{n \to\infty}\int \frac{1}{n}H_\mu\left(\bigvee_{i=0}^{n-1}\varphi^{-1}(i,\o)\a\right)dP(\o),\]
	where $\varphi(i,\o):=\varphi(\t^{i-1}\o)\circ\ldots\circ\varphi(\o)$. The existence of this limit was proved by Abramov and Rokhlin \cite{AR}.  
	Define the fiber entropy of $\varphi$ by 
	\[h_{\widetilde{\mu}}(\varphi):=\sup_{\alpha \in \mathcal{P}_{X}}h_{\widetilde{\mu}}(\varphi,\a).\]
	Moreover, they established the relation between fiber entropy of $\varphi$, and entropy of $\T$ and $\t$, called Abramov-Rokhlin formula, namely, 
	\begin{equation}\label{eq:AR formula}
		h_{\widetilde{\mu}}(\T)=h_{\widetilde{\mu}}(\varphi)+h_P(\t).
	\end{equation}
	
	In fact, the fiber entropy also can be defined as follows 	(See e.g. \cite[Page 384]{MR2186245}).  Given $\widetilde{\a}\in\mathcal{P}_{\O\times X}$, let
	\[h_{\widetilde{\mu}}(\varphi,\widetilde{\a}):=\lim_{n \to\infty}\int \frac{1}{n}H_\mu\left(\bigvee_{i=0}^{n-1}\varphi^{-1}(i,\o)\widetilde{\a}_{\t^i\o}\right)dP(\o),\]
	where $\widetilde{\a}_{\t^i\o}:=\{A_{\t^i\o}:A\in\widetilde{\a}\}$, and $A_\o:=\{x\in X:(\o,x)\in A\}$ for any $\o\in\O$. Then 
	\[h_{\widetilde{\mu}}(\varphi)=\sup_{\widetilde{\a}\in\mathcal{P}_{\O\times X}}h_{\widetilde{\mu}}(\varphi,\widetilde{\a}).\]
	
	\subsection{Disintegration of measures}
	Let $X$ be a compact metric space, $\mathcal{B}_X $ the Borel $\sigma$-algebra of $X$
	and $\mu\in M(X)$.  Let $\mathcal{F}$ be a sub-$\sigma$-algebra of $\mathcal{B}_{X}$. We always assume that $\mathcal{F}$ is the completion under $\mu$.  Then there exists  a unique measurable map $X\to \mathcal{M}(X)$, $x\mapsto \mu_x$ such that for every $f\in L^1(X,\mathcal{B}_{X},\mu)$, $\mathbb{E}_{\mu}(f|\mathcal{F})(x)=\int fd\mu_x$ for $\mu$-a.e. $x\in X$, where $\mathbb{E}_{\mu}(f|\mathcal{F})$ is the conditional expectation of $f$ with respect to the sub-$\sigma$-algebra $\mathcal{F}$. This map is called the disintegration of $\mu$ with respect to $\mathcal{F}$. See \cite[Theorem 5.14]{Wardbook} for more details.
	\subsection{Stable and Unstable sets under $\mathbb{Z}$-actions}
	For a $\mathbb{Z}$-t.d.s.,  the stable set of a point $x\in X$ is defined as
	$$
	W^{s}(x, T)=\left\{y \in X: \lim _{n \to\infty} d\left(T^{n} x, T^{n} y\right)=0\right\}
	$$
	and the unstable set of $x\in X$ is defined as
	$$
	W^{u}(x, T)=\left\{y \in X: \lim _{n \to\infty} d\left(T^{-n} x, T^{-n} y\right)=0\right\} .
	$$
	Clearly, $W^{s}(x, T)=W^{u}\left(x, T^{-1}\right)$ and $W^{u}(x, T)=W^{s}\left(x, T^{-1}\right)$ for each $x \in X$.
	The following result  was proved in \cite[Lemma 3.1]{HLY}.
	\begin{lem}\label{lem4}
		Let $(X, T)$ be a $\mathbb{Z}$-t.d.s. and $\mu \in M(X, T)$ with $h_{\mu}(T)>0$. If
		$
		\mu=\int_{X} \mu_{x} d \mu(x)
		$
		is the disintegration of $\mu$ over the Pinsker algebra $\mathcal{P}_{\mu}(T)$, then for $\mu$-a.e. $x \in X$,
		$$
		\overline{W^{s}(x, T) \cap \operatorname{supp}\mu_{x}}=\operatorname{supp}\mu_{x} \quad \text { and } \quad \overline{W^{u}(x, T) \cap \operatorname{supp}\mu_{x}}=\operatorname{supp}\mu_{x}.
		$$
	\end{lem}
	\section{Directional entropy and Pinkser algebra}
	In this section, we provide an equivalent definition of directional measure-theoretic entropy, which is formally closer to the definition of entropy. With this definition in hand, we prove the directional measure-theoretic entropy of a $\mathbb{Z}^2$-system is equal to the entropy of a skew-product system we constructed.
	
	\smallskip
	Firstly, we recall the classical definition of directional measure-theoretical entropy.	Let $(X,T)$ be a $\mathbb{Z}^2$-t.d.s.,  $\mu\in M(X,T)$ and $\vec{v}=(1,\beta)\in\mathbb{R}^2$  be a direction vector. 
	For any $\alpha\in\mathcal{P}_X$, the directional  measure-theoretic entropy  is defined as
	$$
	\widetilde h^{\vec{v}}_\mu(T, \alpha)=\sup _{B} \limsup_{t \to \infty} \frac{1}{t} H_{\mu}\left(\bigvee_{(m, n) \in B+[0, t] \cdot \vec{v}} T^{-(m,n)}\alpha\right),
	$$
	where $B+[0, t] \cdot \vec{v}=\left\{(m, n) \in \mathbb{Z}^{2} :\right.$ there exists $(l, k) \in B$ such that $(m, n)-(l, k)=p \vec{v}$ for some $p \in[0, t]\}$. We take the supremum over all bounded subsets $B$ of $\mathbb{R}^{2}$. 
	
	We provide an equivalent definition of  directional  measure-theoretic entropy by $$h^{\vec{v}}_\mu(T, \alpha):=\sup_{b>0}\limsup_{N\to\infty}\frac{1}{N}H_{\mu}\left(\bigvee_{(m, n) \in\Lambda_N^{\vec{v}}(b) } T^{-(m,n)}\alpha\right).$$
	Note that for any bounded subsets $B$ of $\mathbb{R}^2$ and $t\in\mathbb{R}$, there exist $b\in(0,\infty)$ and $N\in\mathbb{N}$ large enough such that $B+[0, t] \cdot \vec{v}\subset \Lambda_N^{\vec{v}}(b)$; and conversely, for any $b\in(0,\infty)$ and $N\in\mathbb{N}$, $\Lambda_N^{\vec{v}}(b)$ is a finite subset. Thus, two definitions of directional measure-theoretic entropy are equivalent. Namely, for any partition $\alpha\in \mathcal{P}_X$, 
	$$\widetilde h^{\vec{v}}_\mu(T, \alpha)= h^{\vec{v}}_\mu(T, \alpha).$$
	Thus, the directional measure-theoretic entropy of $\mu$ along $\vec{v}$ can be defined by $$h^{\vec{v}}_\mu(T):=\sup _{\alpha \in \mathcal{P}_{X}} h^{\vec{v}}_\mu(T, \alpha).$$

	Now we construct a skew product system, and prove its entropy is equal to the directional entropy. Let $\O=[0,1)$, $\mathcal{C}$
	be the Borel  $\sigma$-algebra on $\O$, and $m$ be the Lebesgue measure on $\O$. As $\O$ can be viewed as a compact Abelian group (one-dimensional torus), we always use the bi-invariant metric $\rho$ on it, i.e., $\rho(x+z,y+z)=\rho(x,y)$ for any $x,y,z\in \O$. Let $R_\b$ be the rotation on $\O$, i.e., $R_\b t=t+\b\pmod 1$ for any $t\in\O$, where $\b$ is the irrational direction.  Define a skew product system $(\O\times X,\mathcal{C}\times \mathcal{B}_X,m\times \mu, \T)$ by 
	\begin{equation}\label{eq:skew product}
		\T:\O\times X\to \O\times X,\text{ }(\o,x)\mapsto (R_\b t,\varphi(1,t)x),
	\end{equation}
	where $\varphi(1,t)=T^{(1,[\b+t])}$, and $[a]$ is the largest natural number less than $a\in\mathbb{R}$. Denote a metric on $\O\times X$ by 
	\[\widetilde{d}((t,x),(s,y)):=\max\{\rho(t,s),d(x,y)\}\text{ for any }(t,x),(s,y)\in \O\times X.\]
	For convenience, in the following, we denote \[\widetilde{X}=\O\times X,\text{ }\widetilde{\mathcal{B}}=\mathcal{C}\times \mathcal{B}_X,\text{ and }\widetilde{\mu}=m\times \mu.\]
	
	Now we prove the directional entropy of $T$ is equal to the skew entropy of $\T$, and fiber entropy of $\varphi$.
	\begin{thm}\label{lem1}
		$h_{\mu}^{\vec{v}}(T)=h_{\widetilde{\mu}}(\T)=h_{\widetilde{\mu}}(\varphi).$
	\end{thm}
	\begin{proof}
		It follows from \cite[Theorem 3]{Park111} that $h_{\widetilde{\mu}}(\T)=h_{m}(R_\b)+h^{\vec{v}}_{\mu}(T)$. It is well known that the irrational rotation on the torus has zero entropy. Thus, $h_{\widetilde{\mu}}(\T)=h^{\vec{v}}_{\mu}(T)$. Moreover, using \eqref{eq:AR formula}, we have that 	$h_{\mu}^{\vec{v}}(T)=h_{\widetilde{\mu}}(\T)=h_{\widetilde{\mu}}(\varphi).$
	\end{proof}

	To delve deeper into the study of  directional entropy, we introduce the directional Pinsker algebra along $\vec{v}$ by
	\[\mathcal{P}_{\mu}^{\vec{v}}:=\{A\in\mathcal{B}_X: h^{\vec{v}}_\mu(T, \{A,A^c\})=0\}.\]
	From the fact that 
	for any $b_1,b_2\in(0,\infty)$ with $b_1>b_2$, there exists a finite subset $C$ of $\mathbb{Z}^2$ such that
	\begin{equation*}
		\Lambda^{\vec{v}}(b_1)\subset \bigcup_{(m',n')\in C}\lk((m',n')+\Lambda^{\vec{v}}(b_2)\re),
	\end{equation*} we have the following characterization of directional Pinsker algebra.
	\begin{prop}\label{prop4}For any $A\in\mathcal{B}_X$, the following two statements are equivalent:
		\begin{enumerate}[(i)]
			\item  $A\in\mathcal{P}_{\mu}^{\vec{v}}$;
			\smallskip
			\item There exists $b\in(0,\infty)$ such that $$h_\mu^{\vec{v},b}(T,\{A,A^c\}):=\limsup_{N\to\infty}\frac{1}{N}H_{\mu}\left(\bigvee_{(m, n) \in\Lambda_N^{\vec{v}}(b) } T^{-(m,n)}\{A,A^c\}\right)=0.$$
		\end{enumerate}
	\end{prop} 
	Now we study the relation between directional Pinsker algebra and Pinsker algebra of the skew product system.
	\begin{prop}\label{prop5}
		Let $(X,T)$ be a $\mathbb{Z}^2$-t.d.s., $\mu\in M(X,T)$, $\vec{v}=(1,\beta)\in\mathbb{R}^2$ be a direction vector and $(\widetilde{X},\widetilde{\mathcal{B}},\widetilde{\mu},\T)$ be the skew product system defined by \eqref{eq:skew product}. 	Given any $u\in \mathbb{R}$, we define $R_u(t,x)=(t+u,x)$ for any $(t,x)\in\widetilde{X}$. Then $R_u(\mathcal{P}_{\widetilde{\mu}})=\mathcal{P}_{\widetilde{\mu}}$, where $\mathcal{P}_{\widetilde{\mu}}$ is the classical Pinsker algebra of $(\widetilde{X},\widetilde{\mathcal{B}},\widetilde{\mu},\T)$.
	\end{prop}
	\begin{proof}
		For any $A\in \mathcal{P}_{\widetilde{\mu}}$, as $h_m(R_\b)=0$, by \eqref{eq:AR formula},  we have that 
		\[0=h_{\widetilde{\mu}}(\varphi,\a)=\lim_{n \to\infty}\int\frac{1}{n} H_\mu\left(\bigvee_{i=0}^{n-1}\varphi^{-1}(i,t)\a_{R_\b^it}\right)dm(t),\]
		where $\a=\{A,A^c\}$. We consider $R_u\a=\{R_uA,R_uA^c\}$. Since $m$ is  the Lebesgue measure and $\mu$ is $T$-invariant, it follows that
		\begin{align*}
			h_{\widetilde{\mu}}(\varphi,R_u\a)&=\lim_{n \to\infty}\int \frac{1}{n}H_\mu\left(\bigvee_{i=0}^{n-1}\varphi^{-1}(i,t+u)\a_{R_\b^it}\right)dm(t)\\
			&\le\lim_{n \to\infty}\int \frac{1}{n}H_\mu\left(\bigvee_{i=0}^{n-1}\varphi^{-1}(i,t)\a_{R_\b^it}\right)+ \frac{1}{n}H_\mu\left(\bigvee_{i=0}^{n-1}\varphi^{-1}(i,t)T^{-(0,1)}\a_{R_\b^it}\right)dm(t)\\
			&=2\lim_{n \to\infty}\int \frac{1}{n}H_\mu\left(\bigvee_{i=0}^{n-1}\varphi^{-1}(i,t)\a_{R_\b^it}\right)dm(t)=0.
		\end{align*}
		Thus, $R_u(\mathcal{P}_{\widetilde{\mu}})\subset\mathcal{P}_{\widetilde{\mu}}$. Similarly, we can prove the other direction.
	\end{proof}
	
	\begin{thm}\label{thm2} Let $(X,T)$ be a $\mathbb{Z}^2$-t.d.s., $\mu\in M(X,T)$, $\vec{v}=(1,\beta)\in\mathbb{R}^2$ be a direction vector and $(\widetilde{X},\widetilde{\mathcal{B}},\widetilde{\mu},\T)$ be the skew-product system defined by \eqref{eq:skew product}. Then \[ \mathcal{C}\times\mathcal{P}_\mu^{\vec{v}}=\mathcal{P}_{\widetilde{\mu}}.\] Moreover,  if $\mu=\int_{X} \mu_{x} d \mu(x)$ is the disintegration of $\mu$ over the directional Pinsker algebra $\mathcal{P}_\mu^{\vec{v}}$ and $\widetilde\mu=\int_{\widetilde X} \widetilde\mu_{(t,x)} d \widetilde\mu(t,x)$ is the disintegration of $\widetilde\mu$ over the Pinsker algebra $\mathcal{P}_{\widetilde\mu}$, then  \[\widetilde\mu_{(t,x)}=\delta_t\times \mu_{x}
		\text{, for } \widetilde{\mu}\text{-a.e. } (t,x)\in \widetilde{X}.\]
	\end{thm}
	\begin{proof}
		If $h_{\widetilde{\mu}}(\T)=0$ then by Theorem \ref{lem1}, one has $h_{\mu}^{\vec{v}}(T)=0$ and thus \[ \mathcal{C}\times\mathcal{P}_\mu^{\vec{v}}=\mathcal{C}\times\mathcal{B}_X=\mathcal{P}_{\widetilde{\mu}}.\]
		So $\widetilde\mu_{(t,x)}=\delta_{t}\times\delta_x=\delta_t\times \mu_{x}$
		for $\widetilde{\mu}$-a.e. $(t,x)\in \widetilde{X}$.
		Thus, in the next proof, we assume that $h_{\widetilde{\mu}}(\T)>0$.
		By Lemma \ref{lem4}, one has for $\widetilde{\mu}$-a.e. $(t,x)\in\widetilde{X}$,
		\begin{equation}\label{7}
			\overline{W^{s}((t,x), \T) \cap \operatorname{supp}\widetilde\mu_{(t,x)}}=\operatorname{supp}\widetilde\mu_{(t,x)}.
		\end{equation}
		Since $(\O,R_\b)$ is isometric, it follows that if
		$\lim_{n\to\infty}\widetilde d(\T^n(t,x),\T^n(s,y))=0$
		then $t=s$.  Hence for any $x\in X$, there exists a subset $F_x$ of $X$ such that for any $t\in[0,1)$,
		\begin{equation}\label{11}
			W^{s}((t,x), \T)=\{t\}\times F_x.
		\end{equation} 
		By \eqref{7},  one has for $\widetilde{\mu}$-a.e. $(t,x)\in \widetilde{X}$,
		\begin{equation*}
			\operatorname{supp}\widetilde\mu_{(t,x)}\subset \{t\}\times F_x.
		\end{equation*}
		In particular, we consider $t=0$. Then there exists  $\widehat\mu_{x}\in M(X)$ such that
		\begin{equation}\label{10}
			\widetilde\mu_{(0,x)}=\d_0\times\widehat\mu_{x}\text{ for } \mu\text{-a.e. }x\in X,
		\end{equation}
		where $\delta_0$ is the Dirac measure at the point $0$. By Proposition \ref{prop5}, for any $u\in\mathbb{R}$, one has 
		$(R_u)_*\widetilde\mu_{(t,x)}=\widetilde\mu_{(t+u,x)}$
		for $\widetilde{\mu}$-a.e. $(t,x)\in\widetilde{X}$, which together with \eqref{10}, implies that
		\begin{equation}\label{22}
			\widetilde\mu_{(t,x)}=(R_t)_*\widetilde\mu_{(0,x)}=(R_t)_*(\d_0\times\widehat\mu_{x})=\d_t\times\widehat\mu_{x},\text{ 	for $\widetilde{\mu}$-a.e. $(t,x)\in\widetilde{X}$.}
		\end{equation}
		By the uniqueness of the measure decomposition, there exists a sub-$\sigma$-algebra $\mathcal{D}$ of $\mathcal{B}_X$ such that \begin{equation*}
			\mathcal{P}_{\widetilde{\mu}}=\mathcal{C}\times \mathcal{D}.
		\end{equation*}
		
		Thus, we only need to prove that $\mathcal{P}_\mu^{\vec{v}}=\mathcal{D}$.	First we prove that $\mathcal{P}_\mu^{\vec{v}}\subset\mathcal{D}$. It suffices to prove that for any $A\in \mathcal{P}_\mu^{\vec{v}}$, $h_{\widetilde{\mu}}(\varphi,\a)=0$, as $h_m(R_\b)=0$ (see the proof of Abramov-Rokhlin formula \cite[Page 257]{Petersenbook}). Indeed, this is obtained from the fact that for each $N\in\mathbb{N}$, the partition $\bigvee_{(i,j)\in\Lambda_N^{\vec{v}}(1)}T^{-(i,j)}\{A,A^c\}$ is finer than $\bigvee_{i=0}^{N-1}\varphi^{-1}(i,t)\{A,A^c\}$, as $(i,[i\b+t])\in \Lambda_N^{\vec{v}}(1)$, for each $t\in\O$, and $i\in\{0,1,\ldots,N-1\}$.
		
		Next we  prove that $\mathcal{D}\subset \mathcal{P}_\mu^{\vec{v}}$.  For any $A\in \mathcal{D}$, one has that $h_{\widetilde{\mu}}(\varphi,\{A,A^c\})=0$, as
		$h_{\widetilde{\mu}}(\T,\{\O\times A,\O\times A^c \})\ge h_{\widetilde{\mu}}(\varphi,\{A,A^c\})$. Since $m$ is ergodic, it follows from subadditive ergodic theorem that there exists $t\in\O$ such that 
		\[0=h_{\widetilde{\mu}}(\varphi,\{A,A^c\})=\lim_{N \to\infty}\frac{1}{N}H_\mu\left(\bigvee_{i=0}^{N-1}\varphi^{-1}(i,t)\{A,A^c\}\right).\]
		It is straightforward to check that $\Lambda_N^{\vec{v}}(1)\subset \{(i,[i\b+t]),(i,[i\b+t]-1),(i,[i\b+t]+1)\}_{i=0}^{N-1}$, which implies that the partition $\bigvee_{j=-1,0,1}\left(\bigvee_{i=0}^{N-1}\varphi^{-1}(i,t)T^{(0,j)}\{A,A^c\}\right)$ is finer than the partition $\bigvee_{(i,j)\in\Lambda_N^{\vec{v}}(1)}T^{-(i,j)}\{A,A^c\}$. Thus, $h_\mu^{\vec{v},1}(T,\{A,A^c\})=0$. By Proposition \ref{prop4}, we have that $A\in\mathcal{P}_\mu^{\vec{v}}$. The proof is completed.
	\end{proof}
	
	\section{Directional mean Li-Yorke chaos}
	In this section, we consider  mean Li-Yorke chaotic phenomenon along some direction in a $\mathbb{Z}^2$-m.p.s with positive directional entropy. Namely,
	\begin{thm}\label{m-thm1}
		Let  $(X,T)$ be a  $\mathbb{Z}^2$-t.d.s. and $\vec{v}=(1,\beta)\in\mathbb{R}^2$ be a direction vector. If there exists $\mu\in M(X,T)$ such that $h_\mu^{\vec{v}}(T)>0$,
		then  $(X,T)$ is multivariant directional mean Li-Yorke chaotic, that is, there exists
		a Mycielski subset $C$ (i.e., a union
		of countably many Cantor sets) of $X$ such that
		for any $b\in(0,\infty)$, integer $k\geq2$ and pairwise distinct points $x_1,x_2,\dotsc,x_k$ in $C$ it holds that
		\begin{equation}\label{m-1}
			\liminf_{N\to\infty}\frac{1}{\#(\Lambda_N^{\vec{v}}(b))}\sum_{(m,n)\in\Lambda_N^{\vec{v}}(b)}\max_{1\leq i<j\leq k} d(T^{(m,n)}x_i,T^{(m,n)}x_j)=0
		\end{equation}
		and
		\begin{equation}\label{m-2}
			\limsup_{N\to\infty}\frac{1}{\#(\Lambda_N^{\vec{v}}(b))}\sum_{(m,n)\in\Lambda_N^{\vec{v}}(b)}\min_{1\leq i<j\leq k} d(T^{(m,n)}x_i,T^{(m,n)}x_j)\geq \eta_b>0,
		\end{equation}
		where $\#(A)$ is the number of elements in the finite set $A$ and $\eta_b$ is a constant only depending on $b\in(0,\infty)$.
	\end{thm}
	
	Now we will prove Theorem \ref{m-thm1} by applying  the corresponding result about mean Li-Yorke chaos for $\mathbb{Z}$-actions \cite[Theorem 1.1]{HLY} on the skew product system $(\widetilde{X},\widetilde{\mathcal{B}},\widetilde{\mu},\T)$. Firstly, we recall that for $\mathbb{Z}$-actions.
	\begin{lem}\label{lem2}
		If a $\mathbb{Z}$-t.d.s. $(X,T)$ has positive topological entropy,
		then it is multivariant mean Li-Yorke chaotic. Namely, there exists
		a  Mycielski subset $C$ of $X$ such that
		for every integer $k\geq2$ and pairwise distinct points $x_1,x_2,\dotsc,x_k$ in $C$ it holds that
		\[\liminf_{N\to\infty}\frac{1}{N}\sum_{n=0}^{N-1}\max_{1\leq i<j\leq k} d(T^nx_i,T^nx_j)=0\]
		and
		\[\limsup_{N\to\infty}\frac{1}{N}\sum_{n=0}^{N-1}\min_{1\leq i<j\leq k} d(T^nx_i,T^nx_j)\ge \eta>0.\]
	\end{lem}
	
	\medskip
	Now we are able to finish the proof.
	\begin{proof}[Proof of Theorem \ref{m-thm1}]
		Since $h_{\mu}^{\vec{v}}(T)>0$, it follows from Theorem \ref{lem1} that $h_{\mu}^{\vec{v}}(T)=h_{\widetilde{\mu}}(\T)>0$, and hence $h_{top}(\T)>0$ by variational principle. According to Lemma \ref{lem2}, there exists a Mycieski subset $\widetilde{C}$ of $\widetilde{X}$ such that for any $k\ge 2$ and pairwise distinct points $(t_1,x_1),(t_2,x_2),\ldots,(t_k,x_k)\in \widetilde{C}$ such that
		\begin{equation}\label{1}
			\liminf_{N\to\infty}\frac{1}{N}\sum_{n=0}^{N-1}\max_{1\leq i<j\leq k} \widetilde{d}(\T^n(t_i,x_i),\T^n(t_j,x_j))=0
		\end{equation}
		and
		\begin{equation}\label{2}
			\limsup_{N\to\infty}\frac{1}{N}\sum_{n=0}^{N-1}\min_{1\leq i<j\leq k} \widetilde{d}(\T^n(t_i,x_i),\T^n(t_j,x_j))\geq \eta>0.
		\end{equation}
		Since the rotation $R_\b$ is isometric on $\O$ with respect to $\rho$, it follows from \eqref{1} that $t_i=t_j$, for any $1\le i,j\le k$. Thus, there exist  a Mycieski subset $C$ of $X$, and $t\in\O$ such that for any $k\ge 2$ and pairwise distinct points $x_1,x_2,\ldots,x_k\in C$ such that
		\begin{equation}\label{3}
			\liminf_{N\to\infty}\frac{1}{N}\sum_{n=0}^{N-1}\max_{1\leq i<j\leq k} {d}(\varphi(n,t)x_i,\varphi(n,t)x_j)=0
		\end{equation}
		and
		\begin{equation}\label{4}
			\limsup_{N\to\infty}\frac{1}{N}\sum_{n=0}^{N-1}\min_{1\leq i<j\leq k} {d}(\varphi(n,t)x_i,\varphi(n,t)x_j)\geq \eta>0.
		\end{equation}
		
		Using \eqref{3} and \eqref{4}, we will finish the proof of Theorem \ref{m-thm1}. We only prove the case of $b=1$, as other cases can be proved similarly.
		
		Firstly,	we show that \eqref{m-1} holds.
		Since $X$ is compact, we assume that $\mathrm{diam}(X)=1$,
		where $\mathrm{diam}(X)=\max\{d(x,y):x,y\in X\}$. As $T$ is continuous, for any $\epsilon\in(0,1)$, there exists  $\delta\in (0,\epsilon/4)$ such that for any $x,y\in X$ with $d(x,y)<\delta$, we have
		\begin{equation}\label{4.1}
			d(T^{(0,i)}x,T^{(0,i)}y)<\epsilon/4
		\end{equation}
		for each $i\in\{1,0,-1\}$.  By \eqref{3}, there exists $N>0$ such that 
		\[\frac{1}{N}\sum_{n=0}^{N-1}\max_{1\leq i<j\leq k} {d}(\varphi(n,t)x_i,\varphi(n,t)x_j)<\delta^2.\]
		Let $\mathcal{A}_t:=\{n\in\mathbb{Z}_+:\max_{1\leq i<j\leq k} {d}(\varphi(n,t)x_i,\varphi(n,t)x_j)\ge \delta\}.$ Then 
		$$\delta\cdot\frac{\#(\mathcal{A}_t\cap [0,N-1])}{N}\leq\frac{1}{N}\sum_{n=0}^{N-1}\max_{1\leq i<j\leq k} {d}(\varphi(n,t)x_i,\varphi(n,t)x_j)<\delta^2. $$ This implies that
		\begin{equation}\label{4.2}
			\frac{\#(\mathcal{A}_t\cap [0,N-1])}{N}<\delta.
		\end{equation} 
		Note that  $$\Lambda_N^{\vec{v}}(1)\subset\{(n,[n\b+t]-1),(n,[n\b+t]),(n,[n\b+t]+1)\}_{n=0}^{N-1}.$$
		Thus, by \eqref{4.1}, given any $m\notin \mathcal{A}_t$, if $n\in\mathbb{N}$ such that $(m,n)\in \Lambda_N^{\vec{v}}(1)$ then 
		\begin{equation}\label{14}
			\max_{1\leq i<j\leq k} {d}(T^{(m,n)}x_i,T^{(m,n)}x_j)<{\epsilon}/{4}.
		\end{equation}
		Therefore, by the assumption $\mathrm{diam}(X)=1$ and $\delta\in(0,\epsilon/4)$, one has
		\begin{align*}
			&\frac{1}{\#(\Lambda_N^{\vec{v}}(1))}\sum_{(m,n)\in\Lambda_N^{\vec{v}}(1)}\max_{1\leq i<j\leq k} d(T^{(m,n)}x_i,T^{(m,n)}x_j)\\
			\overset{\eqref{14}}\leq& \epsilon/4+2\mathrm{diam}(X)\cdot	\frac{\#(\mathcal{A}_t\cap [0,N-1])}{N}
			\overset{\eqref{4.2}}<\epsilon.
		\end{align*}
		Since $\e>0$ is arbitrary, we have proven that \eqref{m-1} holds when $b=1$.
		
		Now we show that \eqref{m-2} holds when $b=1$.
		Note that for any $t\in[0,1)$ and $n\in\mathbb{N}$, $(n,[n\beta+t])\in \Lambda^{\vec{v}}(1)$. Then by \eqref{4} one has
		\begin{equation*}\label{6}
			\begin{split}
				&\limsup_{N\to\infty}\frac{1}{\#(\Lambda_N^{\vec{v}}(1))}\sum_{(m,n)\in\Lambda_N^{\vec{v}}(1)}\min_{1\leq i<j\leq k} d(T^{(m,n)}x_i,T^{(m,n)}x_j)\\
				\geq&\limsup_{N\to\infty}\frac{1}{2N}\sum_{n=0}^{N-1}\min_{1\leq i<j\leq k} {d}(\varphi(n,t)x_i,\varphi(n,t)x_j)\geq \eta/2>0.
			\end{split}
		\end{equation*} 
		The proof is completed.
	\end{proof}
	\section{Directional entropy tuples for a measure}	
	To further study systems with positive entropy, entropy pair for $\mathbb{Z}$-t.d.s. was introduced by Blanchard \cite{B}. Then Huang and Ye \cite{HY} extended this to entropy tuples. Entropy pair  for a measure on a $\mathbb{Z}$-t.d.s. was introduced by Blanchard, Host, Maass, Martinez and Rudolph \cite{BHM}.  Recently, Park and Lee \cite{PL} introduced topological entropy pair for a $\mathbb{Z}^2$-t.d.s. 
	Corresponding to this work, we introduce directional entropy tuples for a measure on a $\mathbb{Z}^2$-t.d.s. and study many their properties in this section. 
	
	Let us begin with some notions.
	Given a $\mathbb{Z}^2$-t.d.s. $(X,T)$ and $\mu\in M(X,T)$, we obtain that a $\mathbb{Z}^2$-m.p.s. $(X,\mathcal{B}_X,\mu,T)$. Let $X^{(n)}$ be the cartesian product of $X$ with itself $n$ times and $T^{(n)}$ be the simultaneous action of $T$ in each coordinate of $X^{(n)}$. The product $\sigma$-algebra of $X^{(n)}$ is denoted by $\mathcal{B}_X^{(n)}$ and its diagonal by $\Delta_n(X)=\{(x,\ldots,x)\in X^{(n)}:x\in X\}.$ 
	\begin{defn}
		Let $(X,T)$ be a $\mathbb{Z}^2$-t.d.s., $\vec{v}=(1,\beta)\in\mathbb{R}^2$  be a direction vector and $n \geqslant 2$.
		\begin{enumerate}
			\item[(a)] A partition $\alpha=\left\{A_{1}, \ldots, A_{k}\right\}\in \mathcal{P}_X$ is said to be admissible with respect to $\left(x_{1}, \ldots, x_{n}\right) \in X^{(n)}$ if for each $1 \leq i \leq k$ there exists $j_{i}$ such that $x_{j_{i}} \notin \overline{A}_{i}$.
			\item[(b)] An $n$-tuple $\left(x_{1}, \ldots, x_{n}\right) \in X^{(n)}$ is called a $\vec{v}$-entropy $n$-tuple for $\mu \in M(X, T)$, if there exist $i\neq j$ such that $x_i\neq x_j$, and for any admissible  partition $\alpha\in\mathcal{P}_X$ with respect to $\left(x_{1}, \ldots, x_{n}\right)$, $h^{\vec{v}}_\mu(T, \alpha)>0$. Denote by $E_{n}^{\mu,\vec{v}}(X, T)$ the set of all $\vec{v}$-entropy $n$-tuples for $\mu$.
		\end{enumerate}
	\end{defn}
	\noindent Similar to the argument of entropy pair, we know that $E_n^{\mu,\vec{v}}(X,T)\cup \Delta_n(X)$ is a closed subset of $X^{(n)}$.

	To study directional entropy $n$-tuples, we define the measure $\lambda^{\vec{v}}_n(\mu)$ on $\mathcal{B}_X^{(n)}$ by  $$\lambda^{\vec{v}}_n(\mu)(\prod_{i=1}^nA_i)=\int_{X}\prod_{i=1}^n\mathbb{E}_{\mu}(1_{A_i}|\mathcal{P}_\mu^{\vec{v}})d\mu,$$ where $\mathcal{P}_\mu^{\vec{v}}$ is the directional Pinsker algebra of $(X,T)$. Following ideas in \cite{HY}, we can prove the following result.
	\begin{lem}\label{lem3}
		Let $(X,T)$ be a $\mathbb{Z}^2$-t.d.s., $\mu\in M(X,T)$ and $\vec{v}=(1,\beta)\in\mathbb{R}^2$ be a direction vector. If $\mathcal{U}=\{U_1,U_2,\ldots,U_n\}$ is a measurable cover of $X$ with $n\geq 2$, then the following statements are equivalent:
		\begin{itemize}
			\item[(a)] $\lambda^{\vec{v}}_n(\mu)\left(\prod_{i=1}^nU_i^c\right)> 0 $;
			\smallskip
			\item[(b)]for any  partition $\alpha\in\mathcal{P}_X$ finer than $\mathcal{U}$ as a cover, $h^{\vec{v}}_{\mu}(T,\alpha)>0$.
		\end{itemize}
		\begin{proof}
			(b) $\Rightarrow$ (a). Suppose to the contrary that for any partition $\alpha\in\mathcal{P}_X$ finer than $\mathcal{U}$ as a cover, $$h^{\vec{v}}_{\mu}(T,\alpha)>0 \quad\text{and}\quad \lambda^{\vec{v}}_n(\mu)(\prod_{i=1}^nU_i^c)= 0.$$ 
			Let $C_i=\{x\in X: \mathbb{E}_{\mu}(1_{U_i^c}|\mathcal{P}_\mu^{\vec{v}})>0\}\in \mathcal{P}_\mu^{\vec{v}}$ for $1\leq i\leq n$. Then $$\mu(U_i^c\setminus C_i)=\int_{C_i^c}\mathbb{E}_{\mu}(1_{U_i^c}|\mathcal{P}_\mu^{\vec{v}})d\mu=0.$$ Put $D_i=C_i\cup (U_i^c\setminus C_i)$.  Then $D_i\in \mathcal{P}_\mu^{\vec{v}}$ and $D_i^c\subset U_i$.
			
			For any $\textbf{s}=(s(1),s(2),\ldots,s(n))\in \{0,1\}^n$, let $D_{\textbf{s}}=\cap_{i=1}^nD_i\left(s(i)\right)$, where $D_i(0)=D_i$ and $D_i(1)=D_i^c$. Set $D_0^j=\left(\cap_{i=1}^nD_i\right)\cap\left(U_j\setminus\cup_{k=1}^{j-1}U_k \right)$ for $1\leq j\leq n$. Consider the measurable partition 
			$$\alpha=\left\{D_\textbf{s}:\textbf{s}\in\{0,1\}^n\setminus\{(0,\ldots,0)\}\right\}\cup\{D_0^1,\ldots,D_0^n\}.$$
			For any $\textbf{s}\in \{0,1\}^n\setminus \{(0,\ldots,0)\}$, we have $s(i)=1$ for some $1\leq i\leq n$, and thus $D_\textbf{s}\subset D_i^c\subset U_i$. By the construction of $D_0^j$,  we have $D_0^j\subset U_j$ for all $1\leq j\leq n$. Thus $\alpha$ is finer than $\mathcal{U}$ as a cover. However, since $\lambda^{\vec{v}}_n(\mu)(\prod_{i=1}^nU_i^c)= 0 $, we deduce $$\mu\left(\cap_{i=1}^nD_i\right)=\mu\left(\cap_{i=1}^nC_i\right)=0.$$ Thus we have $D_0^1,\ldots,D_0^n\in \mathcal{P}_\mu^{\vec{v}}$. It is also clear that $D_\textbf{s}\in \mathcal{P}_\mu^{\vec{v}}$ for all $\textbf{s}\in\{0,1\}^n\setminus\{(0,\ldots,0)\}$ since $D_1,\ldots,D_n\in \mathcal{P}_\mu^{\vec{v}}.$ Therefore, each element of $\alpha$ is $\mathcal{P}_\mu^{\vec{v}}$-measurable and hence $h^{\vec{v}}_{\mu}(T,\alpha)=0,$ which contradicts the hypothesis that $h^{\vec{v}}_{\mu}(T,\alpha)>0$.
			
			(a) $\Rightarrow$ (b). Assume $\lambda^{\vec{v}}_n(\mu)\left(\prod_{i=1}^nU_i^c\right)> 0 $. Without loss of generality, we may assume that any finite measurable partition $\alpha$ which is finer than $\mathcal{U}$ as a cover is of the type $\alpha=\{A_1,\ldots,A_n\}$ with $A_i\subset U_i$ for $1\leq i\leq n$. Let $\alpha$ be such a partition. We observe that 
			$$\int_X\prod_{i=1}^n\mathbb{E}_{\mu}(1_{A_i^c}|\mathcal{P}_\mu^{\vec{v}})d\mu\geq \int_X\prod_{i=1}^n\mathbb{E}_{\mu}(1_{U_i^c}|\mathcal{P}_\mu^{\vec{v}})d\mu=\lambda^{\vec{v}}_n(\mu)(\prod_{i=1}^nU_i^c)> 0.$$
			So $A_j\notin \mathcal{P}_\mu^{\vec{v}}$ for some $1\leq j\leq n$. Thus $$h^{\vec{v}}_{\mu}(T,\alpha)\ge h^{\vec{v}}_{\mu}(T,\{A_j,A_j^c\})>0.$$
			Now we finish the proof of Lemma \ref{lem3}.
		\end{proof}
	\end{lem}
	Taking advantage of Lemma \ref{lem3}, we describe directional entropy tuples via the support of measure $\lambda^{\vec{v}}_n(\mu)$. We remark that the corresponding result for $\mathbb{Z}$-actions is proved in \cite{HY}.
	
	\begin{thm}\label{thm1}
		Let $(X,T)$ be a $\mathbb{Z}^2$-t.d.s., $\mu\in M(X,T)$ and $\vec{v}=(1,\beta)\in\mathbb{R}^2$ be a direction vector. Then for any $n\geq 2$, one has
		$$E_n^{\mu,\vec{v}}(X,T)=\operatorname{supp}(\lambda^{\vec{v}}_n(\mu))\setminus \Delta_n(X).$$
		\begin{proof}
			First we prove that $$E_n^{\mu,\vec{v}}(X,T)\subset \operatorname{supp}(\lambda^{\vec{v}}_n(\mu))\setminus \Delta_n(X).$$ Let $(x_i)_{i=1}^n\in E_n^{\mu,\vec{v}}(X,T)$. It suffices to prove that for any neighborhood $\prod_{i=1}^nU_i$ of $(x_i)_{i=1}^n$, $\lambda^{\vec{v}}_n(\mu)\left(\prod_{i=1}^nU_i\right)> 0.$
			Let $\mathcal{U}=\{U_1^c,U_2^c,\ldots,U_n^c\}$. It is clear that any partition $\alpha\in\mathcal{P}_X$ finer than $\mathcal{U}$ as a cover is an admissible partition with respect to $(x_i)_{i=1}^n$. Therefore,  $h^{\vec{v}}_{\mu}(T,\alpha)>0$. By Lemma \ref{lem3}, we obtain that $\lambda^{\vec{v}}_n(\mu)\left(\prod_{i=1}^nU_i\right)> 0,$ which implies that $(x_i)_{i=1}^n\in \operatorname{supp}(\lambda^{\vec{v}}_n(\mu))\setminus \Delta_n(X)$. As $(x_i)_{i=1}^n$ is arbitrary, $E_n^{\mu,\vec{v}}(X,T)\subset \operatorname{supp}(\lambda^{\vec{v}}_n(\mu))\setminus \Delta_n(X).$
			
			Next we prove that $$\operatorname{supp}(\lambda^{\vec{v}}_n(\mu))\setminus \Delta_n(X)\subset E_n^{\mu,\vec{v}}(X,T).$$  Let $(x_i)_{i=1}^n\in \operatorname{supp}(\lambda^{\vec{v}}_n(\mu))\setminus \Delta_n(X)$. We need to show that for any admissible partition $\alpha=\{A_1,\ldots,A_k\}$ with respect to $(x_i)_{i=1}^n$,  $h^{\vec{v}}_{\mu}(T,\alpha)>0$. Since $\alpha$ is an admissible partition with respect to $(x_i)_{i=1}^n$, there exist closed neighborhoods $U_i$ of $x_i$ for all $1\leq i\leq n$ such that for each $j\in\{1,2,\ldots,k\}$, there exists $i_j\in\{1,2,\ldots,n\}$ with $A_j\subset U_{i_j}^c$. Hence $\alpha$ is finer than $\mathcal{U}=\{U_1^c,U_2^c,\ldots,U_n^c\}$ as a cover. Since $\lambda^{\vec{v}}_n(\mu)\left(\prod_{i=1}^nU_i\right)> 0$, by Lemma \ref{lem3},  $h^{\vec{v}}_{\mu}(T,\alpha)>0$. This implies that $(x_i)_{i=1}^n\in E_n^{\mu,\vec{v}}(X,T)$. As $(x_i)_{i=1}^n$ is arbitrary, $ \operatorname{supp}(\lambda^{\vec{v}}_n(\mu))\setminus \Delta_n(X)\subset E_n^{\mu,\vec{v}}(X,T).$
		\end{proof}
	\end{thm} 
	\section{Directional stable sets and asymptotic n-tuples}
	In this section, we introduce directional stable (unstable) sets, and asymptotic $n$-tuples, and prove our second application, that is, a characterization of the number of directional asymptotic $n$-tuples in a $\mathbb{Z}^2$-system with positive directional entropy.

	We define $\vec{v}$-stable set of a point $x\in X$ as
	\begin{equation*}
		\begin{split}
			W^{\vec{v},b}_s&(x, T)=\{y \in X: \lim _{k \to\infty} d\left(T^{(m_k,n_k)} x, T^{(m_k,n_k)} y\right)=0\\
			&\text{ for any infinite sequence } \{(m_k,n_k)\}_{k=1}^\infty \text{ of }  \Lambda^{\vec{v}}(b)\text{ with } \{m_k\}_{k=1}^\infty \text{ increasing}\}
		\end{split}
	\end{equation*}
	and the $\vec{v}$-unstable set of $x\in X$  as
	\begin{equation*}
		\begin{split}
			W^{\vec{v},b}_u&(x, T)=\{y \in X: \lim _{k \to\infty} d\left(T^{(m_k,n_k)} x, T^{(m_k,n_k)} y\right)=0\\
			&\text{ for any infinite sequence } \{(m_k,n_k)\}_{k=1}^\infty \text{ of }  \Lambda^{\vec{v}}(b) \text{ with } \{m_k\}_{k=1}^\infty \text{ decreasing}\}.
		\end{split}
	\end{equation*}
	By definitions, it is clear that $W^{\vec{v},b}_s(x, T)=W^{\vec{v},b}_u(x, T^{-1})$ and $W^{\vec{v},b}_u(x, T)=W^{\vec{v},b}_s(x, T^{-1})$, where $(T^{-1})^{(m,n)}=T^{-(m,n)}$, $(m,n)\in\mathbb{Z}^2$. The following result shows that the definitions of directional stable and unstable sets are independent of the choice of $b\in(0,\infty)$. Thus, we omit $b$ to write $W^{\vec{v},b}_s(x, T)$ and $W^{\vec{v},b}_u(x, T)$ as $W^{\vec{v}}_s(x, T)$ and $W^{\vec{v}}_u(x, T)$, respectively.
	
	\begin{prop}\label{prop-1}
		For any $b_1,b_2\in(0,\infty)$, \[W^{\vec{v},b_1}_s(x, T)=W^{\vec{v},b_2}_s(x, T)\text{ and }  W^{\vec{v},b_1}_u(x, T)=W^{\vec{v},b_2}_u(x, T).\]
	\end{prop}
	\begin{proof}
		Fix $b_1,b_2\in(0,\infty)$. Without loss of generality, we assume that $b_1>b_2$. Since $W^{\vec{v},b}_u(x, T)=W^{\vec{v},b}_s(x, T^{-1})$ for any $b\in (0,\infty)$, we only need to prove $W^{\vec{v},b_1}_s(x, T)=W^{\vec{v},b_2}_s(x, T)$.
		
		By definition, it is clear that $W^{\vec{v},b_1}_s(x, T)\subset W^{\vec{v},b_2}_s(x, T).$ On the other hand, we can find a  finite subset $C$ of $\mathbb{Z}^2$ such that
		\begin{equation*}
			\Lambda^{\vec{v}}(b_1)\subset \bigcup_{(m',n')\in C}\lk((m',n')+\Lambda^{\vec{v}}(b_2)\re),
		\end{equation*}
		where $(m',n')+\Lambda^{\vec{v}}(b_2):=\{(m'+m,n'+n):(m,n)\in \Lambda^{\vec{v}}(b_2)\}$. Thus, by the continuity of $T$ and the compactness of $X$, one has
		$W^{\vec{v},b_2}_s(x, T)\subset W^{\vec{v},b_1}_s(x, T).$
		This ends the proof of Proposition \ref{prop-1}.
	\end{proof}
	Moreover, we denote the set of all $\vec{v}$-asymptotic n-tuples as
	\begin{equation*}
		\begin{split}
			Asy^{\vec{v}}_n(X,& T)=\{(x_1,\ldots,x_n) \in X^{(n)}: \lim _{k \to\infty} \max_{1\leq i<j\leq n}d\left(T^{(m_k,n_k)} x_i, T^{(m_k,n_k)} x_j\right)=0\\
			&\text{ for any infinite sequence } \{(m_k,n_k)\}_{k=1}^\infty \text{ of }  \Lambda^{\vec{v}}(b)\text{ with } \{m_k\}_{k=1}^\infty \text{ increasing}\}.
		\end{split}
	\end{equation*}
	Similar to the proof of Proposition \ref{prop-1}, the definition of $Asy^{\vec{v}}_n(X, T)$ is also independent of the choice of $b\in(0,\infty)$.
	
	Now, we provide a characterization of directional stable (unstable) sets via the  stable (unstable) sets of the skew-product system we constructed.  
	\begin{lem}\label{lem5}Let $(X,T)$ be a $\mathbb{Z}^2$-t.d.s., $\mu\in M(X,T)$, $\vec{v}=(1,\beta)\in\mathbb{R}^2$ be a direction vector and $(\widetilde{X},\widetilde{\mathcal{B}},\widetilde{\mu},\T)$ be the skew product system defined by \eqref{eq:skew product}. Then for any $(t,x)\in \widetilde{X}$,
		\[W^s((t,x),\T)=\{t\}\times W_s^{\vec{v}}(x,T) \text{ and } W^u((t,x),\T)=\{t\}\times W_u^{\vec{v}}(x,T).\]
	\end{lem}
	\begin{proof}
		By \eqref{11}, one has for any $x\in X$, there exists a subset $F_x$ of $X$ such that for any $t\in\O$, \[W^{s}((t,x), \T)=\{t\}\times F_x.\]
		Hence we only need to prove that $W_s^{\vec{v}}(x,T)= F_x$. Indeed, this can be obtained by a standard discussion, as $X$ is compact, and $T$ is continuous. 
	\end{proof}
	
	By Lemma \ref{lem5}, we can prove a directional version of Lemma \ref{lem4}.
	\begin{thm}\label{cor1}
		Let $(X,T)$ be a $\mathbb{Z}^2$-t.d.s., $\vec{v}=(1,\beta)\in\mathbb{R}^2$ be a direction vector and $\mu\in M(X,T)$ such that $h_{\mu}^{\vec{v}}(T)>0$. If $\mu=\int_{X} \mu_{x} d \mu(x)$
		is the disintegration of $\mu$ over the  directional Pinsker algebra $\mathcal{P}^{\vec{v}}_{\mu}$, then for $\mu$-a.e. $x \in X$,
		$$
		\overline{W^{\vec{v}}_s(x, T) \cap \operatorname{supp}\mu_{x}}=\operatorname{supp}\mu_{x} \quad \text { and } \quad \overline{W^{\vec{v}}_u(x, T) \cap \operatorname{supp}\mu_{x}}=\operatorname{supp}\mu_{x}.
		$$
	\end{thm}
	\begin{proof} 
		We only need to show the first equation since $W^{\vec{v}}_u(x, T)=W^{\vec{v}}_s(x, T^{-1})$. According to Theorem \ref{lem1}, one has $h_{\widetilde{\mu}}(\T)=h_{\mu}^{\vec{v}}(T)>0$.
		Thus, by Lemma \ref{lem4}, for $\widetilde{\mu}$-a.e. $(t,x)\in\widetilde{X}$,
		\begin{equation}\label{12}
			\overline{W^{s}((t,x), \T) \cap \operatorname{supp}\widetilde{\mu}_{(t,x)}}=\operatorname{supp}\widetilde{\mu}_{(t,x)}. 
		\end{equation}
		By Theorem \ref{thm2} one has 
		\begin{equation}\label{32}
			\operatorname{supp}\widetilde\mu_{(t,x)}=\{t\}\times\operatorname{supp}\mu_{x},\text{ for $\widetilde{\mu}$-a.e. $(t,x)\in\widetilde{X}$.}
		\end{equation}
		Combining \eqref{12}, \eqref{32} and Lemma \ref{lem5},
		one has for $\widetilde{\mu}$-a.e. $(t,x)\in\widetilde{X}$,
		\[\overline{\lk(\{t\}\times W^{\vec{v}}_s(x, T)\re) \cap \lk(\{t\}\times\operatorname{supp}{\mu}_{x}\re)}=\{t\}\times\operatorname{supp}{\mu}_{x}, \]
		which implies that for $\mu$-a.e. $x\in X$,
		\[\overline{W^{\vec{v}}_s(x, T) \cap \operatorname{supp}{\mu}_{x}}=\operatorname{supp}{\mu}_{x}. \]
		This finishes the proof of Theorem \ref{cor1}.
	\end{proof}
	The following lemma will be used in the proof of the second application.
	\begin{lem}\label{lem7}
		Let $(X,T)$ be a $\mathbb{Z}^2$-t.d.s. and $\vec{v}=(1,\beta)\in\mathbb{R}^2$ be a direction vector. Given $\mu\in M^e(X,T)$ with $h_{\mu}^{\vec{v}}(T)>0$,  let $\mu=\int_{X} \mu_{x} d \mu(x)$ be the disintegration of $\mu$ over the  directional Pinsker algebra $\mathcal{P}_\mu^{\vec{v}}$. Then for $\mu$-a.e. $x\in X$, $\mu_x$ is non-atomic.
	\end{lem}
	\begin{proof}
		Since $\mu\in M^e(X,T)$, it follows from the proof of Rohlin's skew-product theorem (see for example \cite[Theorem 3.18]{G}) that either for $\mu$-a.e. $x\in X$, $\mu_x$ is non-atomic or there exists $r\in\mathbb{N}$ such that $\mu$-a.e. $x\in X$, $\mu_x=\frac{1}{r} \sum_{i=1}^{r} \delta_{x_{i}}$, with $x_{1}, x_{2}, \ldots, x_{r}$ distinct points of $X$. By the definition of $\mathcal{P}_\mu^{\vec{v}}$, it is $T$-invariant $\sigma$-algebra. Hence if the second case holds, then $\mathcal{P}_\mu^{\vec{v}}=\mathcal{B}_X$, which contradicts the assumption that $h_{\mu}^{\vec{v}}(T)>0$.
	\end{proof}

	Following ideas in \cite[Theorem 1.2 (1)]{HLY}, we state and prove our second application.
	\begin{thm}\label{m-thm2}
		Let $(X ,T)$ be a $\mathbb{Z}^2$-t.d.s. and $\vec{v}=(1,\beta)\in \mb{R}^2$ be a direction vector. Then for any $\mu\in M^e(X,T)$, \[\overline{Asy^{\vec{v}}_n(X, T) \cap E_{n}^{\mu,\vec{v}}(X, T)}\supset E_{n}^{\mu,\vec{v}}(X, T).\]
	\end{thm}
	\begin{proof}
		Given $\mu\in M^e(X,T)$, if $h_{\mu}^{\vec{v}}(T)=0$, there is nothing to be proved, as $E_{n}^{\mu,\vec{v}}(X, T)=\emptyset$.
		So we suppose that $h_{\mu}^{\vec{v}}(T)>0$. Let $
		\mu=\int_{X} \mu_{x} d \mu(x)
		$
		be the disintegration of $\mu$ over the  directional Pinsker algebra $\mathcal{P}_{\mu}^{\vec{v}}(T)$. Then by Theorem \ref{cor1} and Lemma \ref{lem7}, there exists $X_{1} \in \mathcal{B}_{X}$ with $\mu\left(X_{1}\right)=1$ such that for each $x \in X_{1}$, \[\overline{W^{\vec{v}}_s(x, T) \cap \operatorname{supp}\mu_{x}}=\operatorname{supp}\mu_{x},\text{ and }\mu_x\text{ is non-atomic}.\] It is clear that $W^{\vec{v}}_s(x, T)^{(n)} \subset A s y_{n}^{\vec{v}}(X, T)$ and thus \[As y_{n}^{\vec{v}}(X, T) \cap (\operatorname{supp}\mu_{x})^{(n)} \supset W^{\vec{v}}_s(x, T)^{(n)} \cap( \operatorname{supp}\mu_{x})^{(n)} =\left(W^{\vec{v}}_s(x, T) \cap \operatorname{supp}\mu_{x}\right)^{(n)}.\] This implies that 	for each $x \in X_{1}$,
		\begin{equation}\label{15}
			\overline{A s y_{n}^{\vec{v}}(X, T) \cap \left(\operatorname{supp}\mu_{x}\right)^{(n)}}=\left(\operatorname{supp}\mu_{x}\right)^{(n)} .
		\end{equation}
		By Theorem \ref{thm1}, $E_{n}^{\mu,\vec{v}}(X, T)=\operatorname{supp}\lambda^{\vec{v}}_{n}(\mu) \setminus \Delta_{n}(X)$, and hence \[\lambda_{n}^{\vec{v}}(\mu)\left(E_{n}^{\mu,{\vec{v}}}(X, T) \cup\Delta_{n}(X)\right)=1,\] which implies that there exists $X_{2} \in \mathcal{B}_{X}$ with $\mu\left(X_{2}\right)=1$ such that for each $x \in X_{2}$,  $\mu_{x}^{(n)}\left(E_{n}^{\mu,\vec{v}}(X, T) \cup \Delta_{n}(X)\right)=1$ by the definition of $\lambda^{\vec{v}}_{n}(\mu)$. Thus, 	for each $x \in X_{2}$, $$(\operatorname{supp}\mu_{x})^{(n)}=\operatorname{supp}\mu_{x}^{(n)}\subset E_{n}^{\mu,\vec{v}}(X, T) \cup \Delta_{n}(X),$$ that is, $(\operatorname{supp}\mu_{x})^{(n)} \setminus \Delta_{n}(X) \subset E_{n}^{\mu,\vec{v}}(X, T)$. Now by \eqref{15} and the fact that $\mu_{x}$ is non-atomic, we have for each $x \in X_{1} \cap X_{2}$
		$$
		\overline{A s y_{n}^{\vec{v}}(X, T) \cap E_{n}^{\mu,\vec{v}}(X, T)} \supset \overline{{Asy}^{\vec{v}}_{n}(X, T) \cap(\left(\operatorname{supp}\mu_{x}\right)^{(n)} \setminus \Delta_{n}(X))}=\left(\operatorname{supp}\mu_{x}\right)^{(n)},
		$$
		which implies that
		$$
		\overline{Asy^{\vec{v}}_n(X, T) \cap E_{n}^{\mu,\vec{v}}(X, T)} \supset \bigcup_{x \in X_{1} \cap X_{2}} \left(\operatorname{supp}\mu_{x}\right)^{(n)} .
		$$
		Therefore,
		$$
		\lambda_{n}^{\vec{v}}(\mu)\left(\overline{Asy^{\vec{v}}_n(X, T) \cap E_{n}^{\mu,\vec{v}}(X, T)}\right) \ge \int_{X} \mu_{x}^{(n)}\left(\bigcup_{x \in X_{1} \cap X_{2}} \left(\operatorname{supp}\mu_{x}\right)^{(n)}\right) d \mu(x)=1,
		$$
		which implies that
		\begin{equation*}
			\overline{Asy^{\vec{v}}_n(X, T) \cap E_{n}^{\mu,\vec{v}}(X, T)} \supset \operatorname{supp}\lambda_{n}^{\vec{v}}(\mu) \supset E_{n}^{\mu,\vec{v}}(X, T),
		\end{equation*}
		and hence $Asy^{\vec{v}}_n(X, T) \cap E_{n}^{\mu,\vec{v}}(X, T)$ is dense in $E_{n}^{\mu,\vec{v}}(X, T)$.
	\end{proof}

	\section*{Acknowledgement}
	C. Liu was partially supported by NNSF of China (12090012, 12031019, 12090010). L. Xu was partially supported by NNSF of China (12031019, 12371197).

	\bibliography{aipsamp}
	
\end{document}